\documentclass[12pt]{amsart}

\usepackage{amsmath,amssymb,amsthm,amscd,verbatim, enumerate}
\usepackage{graphicx}
\usepackage{epsfig}
\usepackage{hyperref}
\usepackage{a4wide}

\newtheorem{thm}{Theorem}
\newtheorem{lem}[thm]{Lemma}
\newtheorem{obs}[thm]{Observation}
\newtheorem{claim}[thm]{Claim}
\newtheorem{conj}[thm]{Conjecture}

\newtheorem*{thm*}{Theorem}

\theoremstyle{definition}

\theoremstyle{remark}

\newcommand{\NOS}{\mathbb{N}}
\newcommand{\OS}{\mathbb{S}}

\newcommand{\eg}{\mathbf{eg}}

\def\ceil#1{\left\lceil #1\right\rceil}

\renewcommand{\le}{\leqslant}
\renewcommand{\leq}{\leqslant}
\renewcommand{\ge}{\geqslant}
\renewcommand{\geq}{\geqslant}

\begin{document}

\title{Improper coloring of graphs on surfaces}

\author{Ilkyoo Choi}
\address{Department of Mathematics, Hankuk University of Foreign Studies, Yongin-si, Gyeonggi-do, Republic of Korea}   
\email{ilkyoo@hufs.ac.kr}
\thanks{Ilkyoo Choi is supported by the Basic Science Research Program through the National Research Foundation of Korea (NRF) funded by the Ministry of Education (NRF-2018R1D1A1B07043049), and also by Hankuk University of Foreign Studies Research Fund.
}

\author{Louis Esperet}
\address{Laboratoire G-SCOP (CNRS, Universit\'e Grenoble-Alpes), Grenoble, France}
\email{louis.esperet@grenoble-inp.fr}
\thanks{Louis Esperet is partially supported by ANR Projects STINT
  (\textsc{anr-13-bs02-0007}) and GATO (\textsc{anr-16-ce40-0009-01}), and LabEx PERSYVAL-Lab (\textsc{anr-11-labx-0025}).}

\date{}
\sloppy

\begin{abstract}
A graph $G$ is $(d_1,\ldots,d_k)$-colorable if its vertex set can be partitioned into $k$ sets $V_1,\ldots,V_k$, such that for each $i\in\{1, \ldots, k\}$, the subgraph of $G$ induced by $V_i$ has maximum degree at most $d_i$. 
The Four Color Theorem states that every planar graph is $(0,0,0,0)$-colorable, and a classical result of Cowen, Cowen, and Woodall shows that every planar graph is $(2,2,2)$-colorable. 
In this paper, we extend both of these results to graphs on surfaces. 
Namely, we show that every graph embeddable on a surface of Euler
genus $g>0$ is $(0,0,0,9g-4)$-colorable and
$(2,2,9g-4)$-colorable. Moreover, these graphs are also $(0,0,O(\sqrt{g}),O(\sqrt{g}))$-colorable and
$(2,O(\sqrt{g}),O(\sqrt{g}))$-colorable.
We also prove that every triangle-free graph that is embeddable on a surface of Euler genus $g$ is $(0, 0, O(g))$-colorable.
This is an extension of Gr\"{o}tzsch's Theorem, which states that triangle-free planar graphs are $(0, 0, 0)$-colorable. Finally, we prove that every graph of girth at least 7 that is embeddable on a surface of Euler genus $g$ is $(0,O(\sqrt{g}))$-colorable. All these results are best possible in several ways as the girth condition is sharp, the constant maximum degrees cannot be improved, and the bounds on the maximum degrees depending on $g$ are tight up to a constant multiplicative factor.
\end{abstract}

\maketitle

\section{Introduction}

For a sequence $(d_1,d_2,\ldots,d_k)$ of $k$ integers, we say that a graph $G$ is \emph{$(d_1,d_2,\ldots,d_k)$-colorable} if each vertex of $G$ can be assigned a color from the set $\{1,2,\ldots,k\}$ in such a way that for each $i\in \{1, \ldots, k\}$, a vertex colored $i$ has at most $d_i$ neighbors colored $i$.
In other words, each color class $i$ induces a subgraph of maximum degree at most $d_i$.
Note that a proper coloring is the same as a $(0,0,\ldots,0)$-coloring.
For an integer $d$, a $(d,d,\ldots,d)$-coloring is sometimes called a \emph{$d$-improper coloring} or \emph{$d$-defective coloring}. 

The Four Color Theorem~\cite{AH77a,AH77b} states that every planar graph is $(0,0,0,0)$-colorable, and it was proved by Cowen, Cowen, and Woodall~\cite{CCW86} that every planar graph is also $(2,2,2)$-colorable. 
For any integer $k$, it is not difficult to construct a planar graph that is not $(k, k)$-colorable; one can even find such planar graphs that are triangle-free (see \cite{Skr99}).

A natural question to ask is how these results can be extended to graphs embeddable on surfaces with higher (Euler) genus. 
Cowen, Cowen, and Woodall~\cite{CCW86} proved that every graph of Euler genus $g$ is $(c_4,c_4,c_4,c_4)$-colorable with $c_4=\max \{14,\tfrac13(4g-11)\}$, and conjectured that the same should hold with three colors instead of four.
This was proved by Archdeacon~\cite{Arc87}, who showed that every graph of Euler genus $g$ is $(c_3,c_3,c_3)$-colorable with 
$c_3=\max \{15,\tfrac12(3g-8)\}$. 
The value $c_3$ was subsequently improved to $\max\{12,6+\sqrt{6g}\}$ by Cowen, Goddard, and Jesurum~\cite{CGJ97}, and eventually to $\max\{9,2+\sqrt{4g+6}\}$ by Woodall~\cite{Woo11}.

In this paper, we will show that in the original result of Cowen, Cowen, and
Woodall~\cite{CCW86}, it suffices that only one of the four color classes is not a stable set. 
Namely, we will prove that every graph that is embeddable on a surface of Euler genus $g>0$ is both $(0,0,0,9g-4)$-colorable and $(2,2,9g-4)$-colorable. 
These come as natural extensions of the fact that planar graphs are $(0,0,0,0)$-colorable and $(2,2,2)$-colorable. 
Interestingly, there is a constant $c_1>0$ such that the bound $9g-4$ in these results cannot be replaced by $c_1\cdot g$, so there is no hope to obtain a bound of the same order as $c_3$ above.
In other words, the growth rate of the bound $9g-4$ cannot be improved to a sublinear function of $g$ in both results.

However, when two color classes are allowed to have non-constant maximum degrees, we show that the bound $9g-4$ can be improved to a sublinear function of $g$ in both results. 
Namely, any graph embeddable on a surface of Euler genus $g$ is both $(0,0,K_1(g),K_1(g))$-colorable and $(2, K_2(g), K_2(g))$-colorable with $K_1(g)=20+\sqrt{48g+481}$ and $K_2(g)=38+\sqrt{84g+1682}$.
We also show that the growth rate of $K_1(g)$ and $K_2(g)$ are tight in terms of $g$. 

A famous theorem of Gr\"{o}tzsch~\cite{G59} states that every triangle-free planar graph is 3-colorable. In this paper, we prove that this can be extended to graphs embeddable on surfaces as follows: 
every triangle-free graph embeddable on a surface of Euler genus $g$ is $(0, 0, K_3(g))$-colorable where $K_3(g)=\lceil{10g+32\over 3}\rceil$. 
We prove that $K_3(g)$ cannot be replaced by a sublinear function of $g$, even for graphs of girth at least 6.
It was proved by \v{S}krekovski~\cite{Skr99} that for any $k$, there exist triangle-free planar graphs that are not $(k,k)$-colorable. 
This shows that there does not exist any 2-color analogue of our result on triangle-free graphs on surfaces.

Choi, Choi, Jeong, and Suh~\cite{CCJS14} proved that every graph of girth at least 5 embeddable on a surface of Euler genus $g$ is $(1, K_4(g))$-colorable where $K_4(g)=\max\{10,\lceil {12g+47\over 7}\rceil\}$. 
They also show that the growth rate of $L(g)$ cannot be replaced by a sublinear function of $g$.
On the other hand, for each $k$, Borodin, Ivanova, Montassier, Ochem, and Raspaud~\cite{BIMOR10} constructed a planar graph of girth 6 that is not $(0,k)$-colorable. 

Finally, we prove that every graph of girth at least 7 embeddable on a surface of Euler genus $g$ is $(0,5+\ceil{\sqrt{14g+22}})$-colorable.
On the other hand, we show that there is a constant $c_2>0$ such that for infinitely many values of $g$, there exist graphs of girth at least 7 embeddable on a surface of Euler genus $g$, with no $(0,\lfloor c_2 \sqrt{g}\rfloor)$-coloring.

\smallskip

The results of this paper together with the aforementioned results completely solve\footnote{up to a constant multiplicative factor for
  the maximum degrees $d_i$, depending on $g$.} the following problem: \emph{given
integers $\ell\le 7$, $k$, and $g$, find the smallest $k$-tuple $(d_1,\ldots,d_k)$ in
lexicographic order, such that every
graph of girth at least $\ell$ embeddable on a surface of Euler genus $g$ is $(d_1,\ldots,d_k)$-colorable.}

\section{Preliminaries}

\subsection{Graphs on surfaces}

All graphs in this paper are simple, which means without loops and multiple edges. 
In this paper, a {\em surface} is a non-null compact connected 2-manifold without boundary.  
We refer the reader to the monograph of Mohar and Thomassen~\cite{MoTh} for background on graphs on surfaces.

A surface is either orientable or non-orientable. 
The \emph{orientable surface~$\OS_h$ of genus~$h$} is obtained by adding $h\ge0$
\emph{handles} to the sphere, and the \emph{non-orientable surface~$\NOS_k$ of genus~$k$} is formed by adding $k\ge1$ \emph{cross-caps} to the sphere. 
The {\em Euler genus} $\eg(\Sigma)$ of a surface $\Sigma$ is defined as twice its genus if $\Sigma$ is orientable, and as its genus if $\Sigma$ is non-orientable. 

We say that an embedding is \emph{cellular} if every face is homeomorphic to an open disc of~$\mathbb{R}^2$. 
Euler's Formula states that if~$G$ is a graph with a cellular embedding in a surface~$\Sigma$, with vertex set~$V$, edge set~$E$, and face set~$F$, then $|V|-|E|+|F|\:=\:2-\eg(\Sigma)$.

If $f$ is a face of a graph~$G$ cellularly embedded in a surface~$\Sigma$, then a \emph{boundary walk of~$f$} is a walk consisting of vertices and edges as they are encountered when walking along the whole boundary of~$f$, starting at some vertex and following some orientation of the face. 
The \emph{degree of a face~$f$}, denoted~$d(f)$, is the number of edges on a boundary walk of~$f$.
Note that some edges may be counted more than once.

Let $G$ be a graph embedded in a surface $\Sigma$. 
A cycle $C$ of $G$ is said to be {\em non-contractible} if $C$ is non-contractible as a closed curve in $\Sigma$. 
Also, $C$ is called {\em separating} if $C$ separates $\Sigma$ in two connected pieces, otherwise $C$ is \emph{non-separating}. It is well known that only three types of non-contractible cycles exist (see~\cite{MoTh}): 2-sided separating cycles, 2-sided non-separating cycles, and 1-sided cycles (the latter only appear in non-orientable surfaces, and are non-separating).

The following fact, which is often called the \emph{3-Path  Property}, will be used: 
if $P_{1}, P_{2},P_{3}$ are three internally disjoint paths with the same endpoints in $G$, and $P_{1} \cup P_{2}$ is a non-contractible cycle, then at least one of the two cycles $P_{1} \cup P_{3}$, $P_{2} \cup P_{3}$ is also non-contractible; see for instance~\cite[Proposition 4.3.1]{MoTh}.

We will need the following simple observation about shortest non-contractible cycles. 
The proof presented here is due to Gwena\"el Joret.

\begin{obs}\label{obs:shortest}
Let $G$ be a graph embedded on some surface.
If $C$ is a shortest non-contractible cycle in $G$, then $C$ is an induced cycle and each vertex of $G$ has at most 3 neighbors in $C$.
\end{obs}

\begin{proof}
It is easy to see that if $C$ has a chord, then by the 3-Path Property, $G$ contains a non-contractible cycle shorter than $C$ (recall that $G$ is simple), a contradiction. 
This shows that $C$ is an induced cycle, and in particular, every vertex of $C$ has at most 2 neighbors in $C$.

 Assume now that some vertex $v$ not in $C$ has $k\ge 4$ neighbors in $C$ (in particular, $C$ contains at least 4 vertices). 
Each subpath of $C$ whose end points are neighbors of $v$ and whose internal vertices are not adjacent to $v$ is called a \emph{basic subpath} of $C$. 
Note that the edges of $C$ are partitioned into $k$ basic subpaths of $C$. 
Since $v$ has at least 4 neighbors in $C$, each basic subpath contains at most $|C|-3$ edges. 
A \emph{basic cycle} is obtained from a basic subpath $P$ of $C$ with endpoints $u,w$ by adding the vertex $v$ and the edges $vu$ and $vw$, which are the \emph{rays} of the basic cycle. 

The embedding of $G$ gives an order on the edges incident to $v$. 
If the rays of some basic cycle are not consecutive (among the rays of basics cycles) in the order around $v$, then this basic cycle cannot bound a region homeomorphic to an open disk, and is thus non-contractible. 
Since this basic cycle has length at most $|C|-3+2<|C|$, this contradicts the minimality of $C$.
 We can therefore assume the two rays of each basic cycle are consecutive in the order around $v$, and each basic cycle bounds a region homeomorphic to an open disk. 
By gluing these $k$ regions together, we obtain that $C$ bounds a region homeomorphic to an open disk, which contradicts the fact that $C$ is non-contractible.
\end{proof}

\subsection{Coloring Lemmas}

Let $K\ge 1$, and $k>j\ge 1$ be three integers. Let $d_1,d_2,\ldots,d_k$ be such that
$d_1=\cdots=d_j=K$ and $\max\{d_{j+1},\ldots,d_k\}<K$. In this
section, we study the properties of a graph $G_{j, k}$ that is not $(d_1, \ldots,
d_k)$-colorable, while all its induced subgraphs are $(d_1, \ldots,
d_k)$-colorable.

Let $c_1, \ldots, c_k$ be the $k$ colors of a $(d_1, \ldots, d_k)$-coloring $\varphi$ such that the maximum degree of the graph induced by the color $c_i$ is at most $d_i$ for $i\in\{1, \ldots, k\}$. 
A vertex $v$ is {\it $c_i$-saturated} if $\varphi(v)=c_i$ and $v$ has $d_i$ neighbors colored $c_i$.
Note that by definition, a $c_i$-saturated vertex has at least $d_i$ neighbors. 

For any integer $d$, a \emph{$d$-vertex} is a vertex of degree $d$, a
\emph{$d^+$-vertex} is a vertex with degree at least $d$, and a \emph{$d^-$-vertex} is a vertex with degree at most $d$ . 
The same notation applies to faces instead of vertices. 

\begin{lem}\label{lem:vx-degree}
Every $(K+k-1)^-$-vertex of $G_{j, k}$ has at least $j$ neighbors that are $(K+k)^+$-vertices. 
\end{lem}
\begin{proof}
Assume for the sake of contradiction that a $(K+k-1)^-$-vertex $v$ has at most $j-1$ neighbors that are $(K+k)^+$-vertices. 
By hypothesis, $G_{j, k}-v$ has a $(d_1, \ldots, d_k)$-coloring
$\varphi$. Observe that the colors $c_1, \ldots, c_k$ must all appear in the neighborhood of $v$, since otherwise we could extend $\varphi$ to $G_{j, k}$ by coloring $v$ with the missing color from $c_1, \ldots, c_k$.
Since $v$ is adjacent to at most $j-1$ vertices of degree at least
$K+k$, there exists a color $c_\ell$ with $\ell\in\{1, \ldots, j\}$
such that no neighbor of $v$ that is a $(K+k)^+$-vertex is colored
with $c_\ell$. Assume that $v$ has a $c_\ell$-saturated neighbor
$u$. Then $u$ has degree at most $K+k-1$, and thus has at most $K+k-2$
neighbors distinct from $v$. Among these neighbors, color $c_\ell$ has
to appear $K$ times, so there exists a color distinct from $c_\ell$
(there are $k-1$ such colors) that does not appear in the neighborhood of $u$.
Therefore, we can extend the coloring $\varphi$ to all of $G_{j, k}$
by recoloring all $c_\ell$-saturated neighbors of $v$ with colors
distinct from $c_\ell$ and then letting $\varphi(v)=c_\ell$. 
We obtained a $(d_1, \ldots, d_k)$-coloring of $G_{j, k}$, which is a contradiction.
\end{proof}

\begin{lem}\label{lem:num-high}
There are at least $1+\sum_{i=2}^{k}(d_i+1)$ vertices in $G_{j, k}$ that are $(K+k)^+$-vertices. 
\end{lem}
\begin{proof}
Let $H$ be the set of $(K+k)^+$-vertices of $G_{j, k}$, and assume for the sake of contradiction that $|H|\leq \sum_{i=2}^{k}(d_i+1)$. 
Partition $H$ into $k-1$ sets $S_2, \ldots, S_k$ such that $|S_i|\leq d_i+1$ for $i\in\{2, \ldots, k\}$.

Let $\varphi$ be a coloring of the vertices of $H$ obtained by
assigning color $c_i$ to each vertex of $S_i$, for each $i\in\{2, \ldots, k\}$. 
Since each $S_i$ contains at most $d_i+1$ vertices, the maximum degree of
the graph induced by $S_i$ cannot be more than $d_i$, so $\varphi$ is
indeed a $(d_1, \ldots, d_k)$-coloring of the subgraph of $G_{j,k}$
induced by $H$. 
We now extend $\varphi$ to a $(d_1, \ldots, d_k)$-coloring $\varphi'$
of $G_{j, k}$ in the following greedy fashion: consider a fixed
ordering of the vertices in $V(G_{j, k})-H$ and for each vertex $v$ in
this ordering, we do the following: if the neighborhood of $v$ does
not contain some color $c_i$ with $i\ge 2$ then we assign $c_i$ to
$v$, and otherwise we assign $c_1$ to $v$.

To verify that $\varphi'$ is a $(d_1, \ldots, d_k)$-coloring of $G_{j, k}$, we only need to check that the vertices colored with $c_1$ induce a graph of maximum degree at most $d_1+1$. 
Since no vertex in $H$ is colored with $c_1$ we know that every vertex $v$ colored with $c_1$ has degree at most $K+k-1$. 
Also, $v$ must have neighbors colored with $c_2, \ldots, c_k$ by the greedy algorithm. 
Now, $v$ cannot have $K+1$ neighbors colored with $c_1$ since it has degree at most $K+k-1$.
This shows that $\varphi'$ is a $(d_1, \ldots, d_k)$-coloring of $G_{j, k}$, which is a contradiction.
\end{proof}

\subsection{Discharging procedure}\label{subsec:discharging}

When an embedding of a counterexample $G$ is fixed, we can let $F(G)$ denote the set of faces of this embedding. 
We will prove that $G$ cannot exist by assigning an \emph{initial charge} $\mu(z)$ to each $z\in V(G) \cup F(G)$, and then applying a {\it discharging procedure} to end up with {\it final charge} $\mu^*(z)$ at $z$.
The discharging procedure will preserve the sum of the initial charge, yet, we will prove that the final charge sum is greater than the initial charge sum, and hence we find a contradiction to conclude that the counterexample $G$ does not exist.

%
%

\section{Graphs on Surfaces}\label{sec:main}

\subsection{One part with large maximum degree}

Given a connected subgraph $H$ of a graph $G$, let $G/H$ denote the
graph obtained from $G$ by contracting the edges of $H$ into a single
vertex (and deleting loops and multiple edges in the resulting graph).

The proof of the next result uses a technique that is similar to a tool introduced in~\cite{KT12}, yet our presentation is quite different.

\begin{thm}\label{thm:deg}
For every $g\ge 0$, every connected graph $G$ of Euler genus $g$,
  and every vertex $v$ of $G$, the graph $G$ has a connected subgraph $H$ containing $v$, such that $G/H$ is planar and every vertex of $G$ has at most
$\max\{9g-4,1\}$ neighbors in $V(H)$. 
\end{thm}

\begin{proof}
We will prove the theorem by induction on $g\ge 0$. If $g=0$, then $G$ is
planar and the result directly follows by taking $H$ as the subgraph
of $G$ induced by $\{v\}$. In the
remainder, we may thus assume that $g>0$.

Let $G'$ be a connected graph of Euler genus $g'$ with $0\le g' <g$, and let $P$ be a shortest path between two vertices $u$ and $w$ of $G'$. 
Since $P$ is a shortest path, each vertex of $G'$ has at most 3 neighbors in $V(P)$.
Note that the graph $G^*=G'/P$, which is the graph obtained from $G'$ by contracting $P$ into a single vertex $v^*$, has Euler genus at most $g'$. 
If $g'=0$, then both $G'$ and $G^*$ are planar.
If $g'>0$, then by the induction hypothesis, $G^*$ has a connected subgraph $H^*$ containing $v^*$, such that $G^*/H^*$ is planar and every vertex of $G^*$ has at most $9g'-4$ neighbors in $V(H^*)$.
Let $H'$ be the subgraph of $G'$ induced by the vertices of $H^*-v^*$ and $P$. 
Since $H^*$ contains $v^*$, we know that $H'$ is connected, and thus $G'/H'$ is well-defined.
Note that $G'/H'$ is planar. 
For a vertex $x$ of $G'$, if $x$ is on $P$, then $x$ has at most two neighbors in $V(P)$ and at most $9g'-4$ neighbors in $V(H^*)$, and therefore $x$ has at most $9g'-2$ neighbors in $V(H')$. 
Otherwise $x \not\in P$, and $x$ has at most three neighbors in $V(P)$ and at most $9g'-4$ neighbors in $V(H^*)$ (including $v^*$ if $x$ has a neighbor in $P$), and therefore $x$ has at most $9g'-2$ neighbors in $V(H')$.

We proved that for any $0\le g' <g$, any connected graph $G'$ of Euler genus $g'$, and any pair $u,w$ of vertices of $G'$, there is a connected subgraph $H'$ of $G'$ containing $u$ and $w$ such that $G'/H'$ is planar and every vertex of $G'$ has at most $\max\{3,9g'-2\}$ neighbors in $V(H')$.
This shall be used repeatedly in the remainder of the proof and we sometimes call it the \emph{refined induction}.

\smallskip

Given a graph $G$ with positive Euler genus $g$ and a specified vertex
$v$, let $C$ be a shortest non-contractible cycle in some minimum
Euler genus embedding of $G$.
Such a cycle exists, since otherwise $G$ would be embeddable in the plane and we would have $g=0$.

Assume first that  $C$ is a 2-sided separating cycle. By cutting along $C$ (as  described in~\cite[Section 4.2]{MoTh}, for example), we obtain two graphs $G_1$ and $G_2$ embedded in two surfaces $\Sigma_1$ and $\Sigma_2$ of Euler genus $g_1>0$ and $g_2>0$, respectively, such that $g=g_1+g_2$. 
By symmetry, we can assume that $v$ 
lies in $G_1$. 
Note that $C$ corresponds to a face $f_1$ and a face $f_2$ in $G_1$ and $G_2$, respectively. For $i=1,2$, let $G_i^*$ be the graph obtained from $G_i$ by contracting all the vertices incident with $f_i$ into a single vertex $v_i$. 
Note that $G_1^*$ and $G_2^*$ are embeddable on surfaces of Euler genus $g_1$ and $g_2$, respectively. 

By the refined induction hypothesis, there is a connected subgraph $H_1^*$ of $G_1^*$ containing $v$ and $v_1$, such that $G_1^*/H_1^*$ is planar, and every vertex of $G_1^*$ has at most $9g_1-2$ neighbors in $V(H_1^*)$. 
By the induction hypothesis, there is also a connected subgraph $H_2^*$ of $G_2^*$ containing $v_2$, such that $G_2^*/H_2^*$ is planar, and every vertex of $G_2^*$ has at most $9g_2-4$ neighbors in $V(H_2^*)$. 
Let $H$ be subgraph of $G$ induced by the vertices of $H_1^*-\{v_1\}$, $C$, and $H_2^*-\{v_2\}$. 
We know that $H$ is connected and contains $v$. 
Moreover,  $G/H$ is also planar since it is obtained by identifying $v_1$ and $v_2$ from the two planar graphs $G_1^*/H_1^*$ and $G_2^*/H_2^*$.
Since $C$ is a shortest non-contractible cycle, it follows from Observation~\ref{obs:shortest} that $C$ is an induced cycle and each vertex not in $C$ has at most three neighbors in $C$. 
As a consequence, each vertex of $C$ has at most $(9g_1-2)+(9g_2-4)+2=9g-4$ neighbors in $V(H)$, and each vertex not in $C$ has at most $\max\{(9g_1-2)+2,(9g_2-4)+2\}\le 9g-9$ neighbors in $V(H)$. 
Thus, we have obtained a connected subgraph $H$ containing $v$ such that $G/H$ is planar and every vertex of $G$ has at most $9g-4$ neighbors in $V(H)$, as desired.

Assume now that $C$ is a 1-sided cycle. 
By cutting along $C$ we obtain a graph $G'$ embedded in a surface $\Sigma'$ of Euler genus
$g'\in\{0, \ldots, g-1\}$ in which $C$ corresponds to a face $f$. 
Contract all the vertices incident with $f$ into a single vertex $v^*$, and note that the resulting graph $G^*$ can also be embedded in $\Sigma'$. 
By the refined induction hypothesis, $G^*$ has a connected subgraph $H^*$ containing $v$ and $v^*$ such that $G^*/H^*$ is planar and every vertex of $G^*$ has at most $\max\{3,9g'-2\}\le 9g'+3\leq 9g-6$ neighbors in $V(H^*)$. 
Using the same argument  as above, the subgraph $H$ of $G$ induced by
the vertices of $H^*-\{v^*\}$ and $C$ is  connected, $G/H$ is planar, and every vertex of $G$ has at most $9g-4$  neighbors in $V(H)$.

It remains to consider the case when $C$ is a 2-sided non-separating cycle. 
In this case, cutting along $C$ yields a graph $G'$ embeddable on a surface $\Sigma'$ of Euler genus $g'\le g-2$, in which $C$ corresponds to two faces $f_1$ and $f_2$ lying in the same connected component. 
We take a shortest path $P$ between $f_1$ and $f_2$ in $G'$, and then 
contract all the vertices of $P$ and vertices incident with $f_1$ or $f_2$ into a single vertex $v^*$. 
Note that the resulting graph $G^*$ is embeddable on $\Sigma'$. 
By the refined induction, $G^*$ has a connected subgraph $H^*$ containing $v$ and $v^*$ such that $G^*/H^*$ is planar and every vertex of $G^*$ has at most $\max\{3,9g'-2\}\le 9g'+3$ neighbors in $V(H^*)$. 
Let $H$ be the subgraph of $G$ induced by the vertices of $C$, $P$, and $H^*-\{v^*\}$. 
Since $H$ is connected, $G/H$ is well-defined and is therefore planar. 
Let $u$ be a vertex of $G$ not in $C\cup P$. 
Since $C$ is a shortest non-contractible cycle in $G$, by Observation~\ref{obs:shortest} the vertex $u$ has at most 3 neighbors in $C$. 
Since $P$ is a shortest path in $G'$, $v$ has at most 3 neighbors in $V(P)$ and therefore $v$ has at most 6 neighbors in $C\cup P$. 
It follows that $u$ has at most $(9g'+3)+6-1=9g'+8\le 9g-10$ neighbors in $V(H)$. 
By Observation~\ref{obs:shortest}, a vertex $u$ of $C\cup P$ has at most $3+2=5$ neighbors in $C\cup P$. 
It follows that $u$ has at most $(9g'+3)+5\le 9g-10$ neighbors in $V(H)$. 
Consequently, $H$ is a connected subgraph containing $v$ such that $G/H$ is planar and each
vertex of $G$ has at most $9g-4$ neighbors in $V(H)$, as desired.
\end{proof}

We are now able to obtain the two following results as simple consequences of Theorem~\ref{thm:deg}.

\begin{thm}\label{thm:000k}
For each $g> 0$, every graph of Euler genus $g$ has a $(0,0,0,9g-4)$-coloring.
\end{thm}

\begin{proof}
Let $G$ be a graph of Euler genus $g>0$. 
We may assume that $G$ is connected, since we can color each connected component independently (and each of its connected components has Euler genus at most $g$). 
By Theorem~\ref{thm:deg}, $G$ has a connected subgraph $H$ such that $G/H$ is planar and every vertex of $G$ has at most $9g-4$ neighbors in $V(H)$. 
By the Four Color Theorem, $G/H$ has a proper 4-coloring. 
Assume without loss of generality that the vertex $v$ of $G/H$ resulting from the contraction of $H$ has the fourth color. 
We extend this coloring to $G$ by assigning the fourth color to all vertices of $H$. 
Since each vertex of $H$ has at most $9g-4$ neighbors in $V(H)$, and each neighbor of a vertex of $H$ outside $H$ does not have the fourth color, the obtained coloring is indeed a $(0,0,0,9g-4)$-coloring of $G$, as desired.
\end{proof}

\begin{thm}\label{thm:22k}
For each $g> 0$, every graph of Euler genus $g$ has a $(2,2,9g-4)$-coloring.
\end{thm}

\begin{proof}
Let $G$ be a graph of Euler genus $g>0$. 
As before, we may assume that $G$ is connected, since we can color each connected component independently. 
By Theorem~\ref{thm:deg}, $G$ has a connected subgraph $H$ such that $G/H$ is planar and every vertex of $G$ has at most $9g-4$ neighbors in $V(H)$. 
Cowen, Cowen, and Woodall~\cite{CCW86} proved that for every planar graph $G'$ and any specified vertex $v'$ in $G'$, the graph $G'$ has a $(2,2,2)$-coloring in which $v'$ has no neighbor of its color. 
It follows that $G/H$ has a $(2,2,2)$-coloring in which the vertex $v$ of $G/H$ resulting from the contraction of $H$ has no neighbor of its color; 
without loss of generality assume that $v$ has the third color. 
We extend this coloring to $G$ by assigning the third color to all vertices of $H$. 
Since each vertex of $H$ has at most $9g-4$ neighbors in $V(H)$, and each neighbor of a vertex of $H$ outside $H$ does not have the third color, the obtained coloring is indeed a $(2,2,9g-4)$-coloring of $G$, as desired.
\end{proof}

We now prove that Theorems~\ref{thm:000k} and~\ref{thm:22k} are best
possible, up to the multiplicative constant 9. More precisely, we will
show that $9g-4$ cannot be replaced by a sublinear function of $g$ in
Theorems~\ref{thm:000k} and~\ref{thm:22k}.

Given a graph $H$ and an integer $k$, we construct the graph $S(H,k)$ as
follows. We start with a copy of $H$, which we call the \emph{basic
  copy} of $H$. 
For each vertex $v$ in the basic copy, we add $k$ new pairwise
disjoint copies of $H$, and add all possible edges between $v$ and
these $k$ copies of $H$.

Consider the graph $G_1=S(K_4,k+1)$, for some integer $k$. Note that the
blocks of $G_1$ consist of one copy
of $K_4$ and $4(k+1)$ copies of $K_5$. Since the Euler genus of a
graph is the sum of the Euler genera of its blocks (see for instance
Theorem 4.4.3 in~\cite{MoTh}), $G_1$ has Euler genus $4(k+1)$.

Assume for the sake of contradiction that $G_1$ has a
$(0,0,0,k)$-coloring, say with colors $1,2,3,4$, where the fourth
color induces a graph with maximum degree at most $k$. Since colors
$1,2,3$ induce stable sets, at least one of the vertices of the basic
copy of $K_4$, call it $v$, is colored 4. Since $v$ has at most $k$
neighbors colored 4, at least one of the copies of $K_4$ joined to $v$
has no vertex colored 4, and hence is properly colored with $1,2,3$, which is a contradiction. 
It follows that $G_1$ is a graph of Euler genus $g$
with no $(0,0,0,\tfrac{g}{4}-1)$-coloring (and such a
graph can be constructed for infinitely many values of $g$).

\smallskip

We now consider $G_2=S(K_7,k+1)$, for some integer $k$. The blocks of
the graph
$G_2$ consist of one copy of $K_7$ (of Euler genus 2) and $7(k+1)$
copies of $K_8$ (each of Euler genus 4). Therefore, $G_2$ has Euler
genus $28k+30$. If $G_2$ admits some $(2,2,k)$-coloring with colors $1,2,3$,
where colors $1,2$ induce a graph with maximum degree 2 and color 3
induces a graph with maximum degree $k$, then some vertex $v$ of the
basic copy of $K_7$ in $G_2$ has color $3$. As before, $v$ has to be
joined to a copy of $K_7$ in which all the vertices have color 1 or 2,
which is a contradiction. Therefore, we found, for infinitely many values of $g$, a graph with Euler genus
$g$ with no $(2,2,\lceil \tfrac{g-30}{28}\rceil)$-coloring.

By considering the graph $S(K_n,\ell)$, for large $n$ and $\ell$, it
is not difficult to see that for any $k$, there is a constant
$\epsilon>0$ such that we can construct (for infinitely many values of
$g$) graphs of Euler genus $g$ with no $(k,k,\lceil\epsilon\,
g\rceil)$-coloring. 

\smallskip

Note however that if we let the maximum degree of the
second color class be a function of $g$, then the
maximum degree of the third color class can be made sublinear: it can
be derived from the main result of~\cite{Woo11} that every graph of
Euler genus $g$ is $(9,O(\sqrt{g}),O(\sqrt{g}))$-colorable. In the
next subsection, we will prove that every graph of
Euler genus $g$ is $(2,O(\sqrt{g}),O(\sqrt{g}))$-colorable and the
constant 2 there is best possible. It is a folklore result that for any $k$, there exist planar graphs that are not
$(1,k,k)$-colorable. Since we have not been able to find a reference
of this result, we include a construction below for the sake of
completeness. This result implies that Theorems~\ref{thm:000k},
\ref{thm:22k}, \ref{thm:2kk}, and \ref{thm:00kk} cannot be improved by reducing the number of colors,
or the maximum degree of the monochromatic components (except for the
color classes whose degree depends on $g$).

\smallskip

{\bf Construction of planar graphs that are not
  $(1,k,k)$-colorable.} 
In a $(1, k, k)$-coloring, let $1, k_1,k_2$ be
the three colors where the vertices of color $1, k_1, k_2$ induce a
graph of maximum degree at most $1, k, k$, respectively. 
Given a planar graph $G$ and two adjacent vertices $x$ and $y$, by
\emph{thickening} the edge $xy$ we mean adding $2k+1$ pairwise disjoint paths on 3
vertices to $G$, and making all the newly added vertices adjacent to
both $x$ and $y$ (see Figure~\ref{fig-non1kk}, left). Note that this can be done in such way that the
resulting graph $H$ is also planar. We claim that in any $(1, k,
k)$-coloring $c$ of $H$, we do not have
$\{c(x),c(y)\}=\{k_1,k_2\}$. Otherwise, some path on 3 vertices joined
to $x$ and $y$ would not contain colors $k_1$ and $k_2$, and then some
vertex of color 1 would have two neighbors colored 1, a contradiction.

Now, take a cycle $C$ on $3k+1$
vertices, and add a vertex $z$ adjacent to all the vertices of $C$. The
obtained graph $G_z$ is planar. Now, thicken all the edges of $G$
joining $z$ and $C$, and
call the resulting graph $H_z$ (see Figure~\ref{fig-non1kk}, center). We claim that in any $(1, k,k)$-coloring $c$ of $H_z$, $c(z)=1$. Suppose for the sake of
contradiction that $z$ has color $k_1$ or $k_2$ (say $k_1$ by
symmetry). Then at least $2k+1$ of the neighbors of $z$ in $C$ have
color 1 or $k_2$. Since no three consecutive vertices of $C$ have
color 1, at least one vertex $u$ of $C$ has color $k_2$. Since the
edge $uz$ was thickened, this contradicts the previous paragraph.

\begin{figure}[ht]
	\begin{center}
  \includegraphics[scale=0.7]{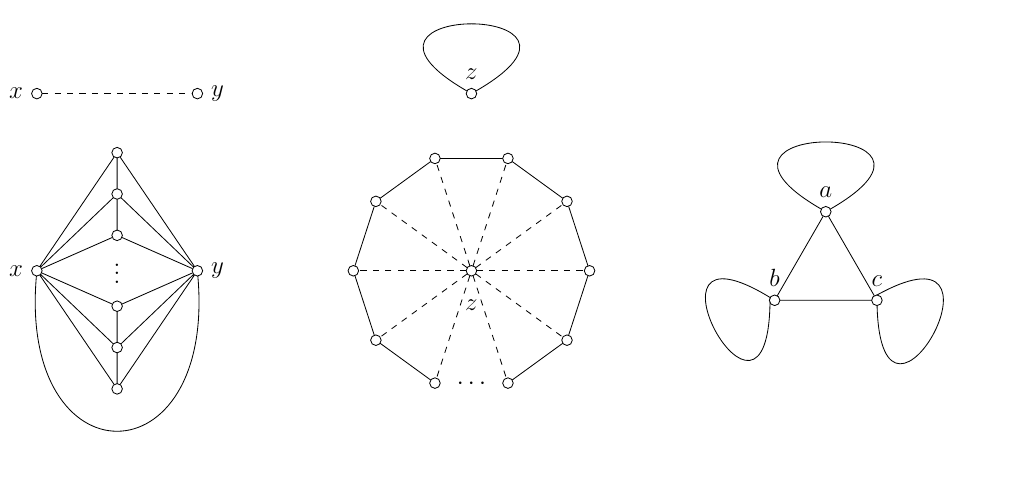}
  \caption{A construction of a planar graph that is not $(1,k,k)$-colorable.}
  \label{fig-non1kk}
	\end{center}
\end{figure}

Our construction now proceeds as follows. Start with a triangle $abc$,
and then identify $a$ with the vertex $a$ of some copy of $H_a$, $b$
with the vertex $b$ of some copy of $H_b$ and $c$ with the vertex $c$ in
some copy of $H_c$ (see Figure~\ref{fig-non1kk}, right). Note that in any $(1, k, k)$-coloring of this graph,
at least one of $a,b,c$ has a color distinct from 1. It then follows from
the previous paragraph that this graph is not $(1, k, k)$-colorable.

\subsection{Two parts with large maximum degrees}

\begin{thm}\label{thm:2kk}
Every graph embeddable on a surface of Euler genus $g$ is $(2, K, K)$-colorable where $K=K(g)=38+\sqrt{84g+1682}$.
\end{thm}
\begin{proof}
Assume for the sake of contradiction that there is a graph $G$ embeddable on a surface of Euler genus $g$ that is not $(2,K,K)$-colorable. We choose $g$ minimum, and with respect to this, we choose $G$ such that the sum of the number of vertices and the number of edges is minimum. By the minimality of $g$ we may assume that $G$ is cellularly embedded on a surface of Euler genus $g$ (see~\cite[Propositions 3.4.1 and 3.4.2]{MoTh}) and from now on, we fix this embedding. 

By the minimality of $G$, we can also assume that $G$ is connected and has minimum degree at least 3.
A \emph{high} and \emph{low} vertex is a vertex of degree at least $K+3$ and at most $K+2$, respectively. 
By Lemma~\ref{lem:vx-degree}, every low vertex is adjacent to at least two high vertices.
By Lemma~\ref{lem:num-high}, $G$ contains at least $K+5$ high vertices. 



\begin{claim}\label{cl:8}
No two vertices of degree at most 4 are adjacent. 
\end{claim}
\begin{proof}
Assume for the sake of contradiction that two vertices $u,v$ of degree at most 4 are adjacent. 
By the minimality of $G$, the graph obtained from $G$ by removing the edge $uv$ has a $(2,K,K)$-coloring $c$. 
Let the three colors be $2,k_1,k_2$ so that the maximum degree of the graph induced by $2,k_1,k_2$ is at most $2,K,K$, respectively.
Since $G$ itself is not $(2,K,K)$-colorable, both $u$ and $v$ are colored 2 and at least one of $u,v$, say $u$, has two neighbors (distinct from $v$) that are also colored 2. 
As a consequence, either $k_1$ or $k_2$ does not appear in the neighborhood of $u$.
We can therefore recolor $u$ with the missing color to get a $(2,K,K)$-coloring of $G$, a contradiction.
\end{proof}

We will use the discharging procedure laid out in Subsection~\ref{subsec:discharging}.
For a vertex $v$ and a face $f$ of $G$, the initial charge is $d(v)-6$ and $2d(f)-6$, respectively. 
The initial charge sum is $6g-12$ by Euler's formula. 

Here are the discharging rules:


\begin{enumerate}[(R1)]

\item Each face distributes its initial charge (evenly) to its incident vertices of degree 3.

\item Each high vertex sends charge $\tfrac{13}{14}$ to each low neighbor. 

\item For each high vertex $v$ and each sequence of three consecutive neighbors $u_1,u_2,u_3$ of $v$ in clockwise order around $v$ such that $u_2$ is high, $v$ sends charge $\tfrac{13}{28}$ to each of $u_1$ and $u_3$.

\item Every low vertex of degree at least 5 sends charge $\tfrac3{14}$ to each neighbor of degree at most 4.
\end{enumerate}

We now analyze the charge of each vertex and each face after the discharging procedure.

Since every face has degree at least 3, every face has nonnegative initial charge and does not send more that its initial charge by (R1). 
Therefore, the final charge of each face is nonnegative.

For a $3$-vertex $v$, let $x,y,z$ be the neighbors of $v$. 
By Lemma~\ref{lem:vx-degree} and Claim~\ref{cl:8}, we may assume without loss of generality that $x,y$ are high and $z$ has degree at least 5. 
First, assume that the face $f$ incident to the edges $vx$ and $vy$ is a triangle, which implies that $x$ and $y$ are adjacent. 
Then $v$ receives charge $\tfrac{13}{14}$ (by (R2)) and $\tfrac{13}{28}$ (by (R3)) from $x$, and the same amount from $y$. Note that $v$ also receives $\tfrac3{14}$ from $z$ by (R4). 
As a consequence, the final charge of $v$ is at least $-3 +2\cdot \tfrac{13}{14}+2 \cdot \tfrac{13}{28}+\tfrac3{14}=0$. 
Now assume that the face $f$ has degree $d\ge 4$. 
Then, $v$ receives charge $\tfrac{13}{14}$ from each of $x$ and $y$ by (R2), and $\tfrac3{14}$ from $z$ by (R4). But since $f$ contains at least two high vertices, it also sends charge at least $\tfrac{2d-6}{d-2}\ge 1$ to $v$ (since $d\ge 4$). 
As a consequence, the final charge of $v$ is at least $-3 +2\cdot \tfrac{13}{14}+\tfrac3{14}+1=\tfrac1{14}$.

Let $v$ be a vertex of degree 4. Since $v$ has at least two high neighbors, and all the neighbors of $v$ have degree at least 5, $v$ receives charge at least $2\cdot \tfrac{13}{14}$ by (R2) and $2\cdot \tfrac{3}{14}$ by (R4). The final charge of $v$ is therefore at least $-2+2\cdot \tfrac{13}{14}+2\cdot \tfrac{3}{14}=\tfrac27$.

Let $v$ be a low vertex of degree $d\ge 5$. 
Then $v$ receives charge at least $2\cdot \tfrac{13}{14}$ by (R2) since it has at least two high neighbors, and sends charge at most $(d-2)\tfrac3{14}$ by (R4). Therefore, the final charge of $v$ is at least $d-6+2\cdot \tfrac{13}{14}-(d-2)\tfrac3{14}\ge \tfrac{3}{14}$.

Finally, let $v$ be a high vertex of degree $d$ (recall by definition, $d\ge K+3$). 
Then $v$ sends charge at most $\tfrac{13}{14}d$ by (R2) and (R3) and its final charge is at least $d-6-\tfrac{13}{14}d=\tfrac{d}{14}-6\ge \tfrac{K-81}{14}$.

Since there are at least $K+5$ high vertices, the total final charge (which equals $6g-12$) is at least $(K+5)\tfrac{K-81}{14}$. We obtain $K^2-76K-84g-237\le 0$, and this contradicts our choice of $K$ since $K$ satisfies $K^2-76K-84g-237=1$.
\end{proof}

We now prove that the order of magnitude of $K(g)$ in Theorem~\ref{thm:2kk} is best possible. For a given $k\ge 0$, we construct the following graph $G_k$. Start with a copy of $K_4$ (which we call the {\it basic copy} of $K_4$), together with $k+1$ other disjoint copies of $K_4$, and add all possible edges between the vertices of the basic copy of $K_4$ and the vertices of the other copies of $K_4$ (but no edge between two non-basic copies of $K_4$). 
These edges are called the {\it support edges} of the construction. 
Now, for each support edge $uv$, create $2k+1$ new disjoint copies of $K_4$ and join $u$ and $v$ to all the newly created vertices. 
Note that the resulting graph $G_k$ has $128k^2+196k+72$ vertices and
$448k^2+694k+252$ edges. It follows from Euler's Formula (and the fact
that any connected graph has a minimum Euler genus embedding that is
cellular) that any connected
graph on $n$ vertices and $m$ edges has Euler genus at most $m-n+1$.
In particular, the graph $G_k$ has Euler genus at most $320k^2+498k+181$.
Consider any $(2,k,k)$-coloring of $G_k$. We adopt the same convention as in the previous proof (the colors are named $2,k_1,k_2$). 
Then at least one of the 4 vertices of the basic copy of $K_4$, call it $u$, is colored $k_1$ or $k_2$, say $k_1$. 
Since $u$ is adjacent to all the vertices in the $k+1$ non-basic copies of $K_4$, at least one of them  contains a vertex $v$ of color $k_2$. 
At most $k$ of the $2k+1$ copies of $K_4$ joined to both $u,v$ contain a vertex colored $k_1$, and at most $k$ of them contain color $k_2$. 
Therefore, at least one copy of $K_4$ contains vertices only colored with 2, which is a contradiction. 
It follows that $G_k$ is not $(2,k,k)$-colorable. 
Consequently, there is a constant $c>0$ and infinitely many values of $g$, for which we can construct a graph embeddable on a surface of Euler genus $g$, with no $(2,\lfloor c\sqrt{g}\rfloor,\lfloor c\sqrt{g}\rfloor)$-coloring.

\smallskip

Note that the same analysis shows that $G_k$ is not $(0,0,k,k)$-colorable. 
We can even replace each copy of $K_4$ by a triangle, and this property remains true. 
Therefore this graph also shows that we can construct, for infinitely many values of $g$, a graph embeddable on a surface of Euler genus $g$, with no $(0,0,\lfloor c\sqrt{g}\rfloor,\lfloor c\sqrt{g}\rfloor)$-coloring. The next result shows that this is also asymptotically best possible.

\begin{thm}\label{thm:00kk}
Every graph embeddable on a surface of Euler genus $g$ is $(0,0,K,K)$-colorable, with $K=K(g)=20+\sqrt{48g+481}$.
\end{thm}

\begin{proof}

Assume for the sake of contradiction that there is a graph $G$ embeddable on a surface of Euler genus $g$ that is not $(0,0,K,K)$-colorable. We choose $g$ minimum, and with respect to this, we choose $G$ such that the sum of the number of vertices is minimum. By the minimality of $g$ we may assume that $G$ is cellularly embedded on a surface of Euler genus $g$ (see~\cite[Propositions 3.4.1 and 3.4.2]{MoTh}) and from now on, we fix this embedding. 

Moreover, we can assume that $G$ is edge-maximal with respect to this embedding (and such that $G$ is simple), since if a supergraph of $G$ can be $(0,0,K,K)$-colored, then $G$ can also be $(0,0,K,K)$-colored. In particular, it follows that for every vertex $v$, there is a circular ordering on the neighbors of $v$ such that any two consecutive vertices in this ordering are adjacent (note that $G$ does not necessarily triangulate the surface it is embedded in).

By the minimality of $G$, we can also assume that $G$ is connected and has minimum degree at least 4.
A \emph{high} and \emph{low} vertex is a vertex of degree at least $K+4$ and at most $K+3$, respectively. 
By Lemma~\ref{lem:vx-degree}, every low vertex is adjacent to at least two high vertices.
By Lemma~\ref{lem:num-high}, $G$ contains at least $K+4$ high vertices. 


\begin{claim}\label{cl:4tri}
Let $v$ be a $4$-vertex with neighbors $u_1,u_2,u_3,u_4$.
If $vu_1u_2$, $vu_2u_3$, $vu_3u_4$, and $vu_4u_1$ are triangular faces, then $u_1u_3$ and $u_2u_4$ are edges in $G$. 
\end{claim}
\begin{proof}
Without loss of generality, assume that $u_1$ and $u_3$ are not adjacent. 
Remove $v$ and identify $u_1$ and $u_3$ into a single vertex.
Note that this can be done in such a way that the resulting graph is still embeddable on the same surface. 
By the minimality of $G$, the resulting graph is $(0,0,K,K)$-colorable and any $(0,0,K,K)$-coloring can easily be extended to $v$ since only three colors appear in its neighborhood, a contradiction. 
\end{proof}

%

We will use the discharging procedure laid out in Subsection~\ref{subsec:discharging}.
For a vertex $v$ and a face $f$ of $G$, the initial charge is $d(v)-6$ and $2d(f)-6$, respectively. 
The initial charge sum is $6g-12$ by Euler's formula. 

Here are the discharging rules:

\begin{enumerate}[(R1)]
\item Each face of degree at least 4 sends charge $\tfrac14$ to each incident vertex of degree 4.

\item Each high vertex sends charge $\tfrac78$ to each low neighbor. 

\item For each high vertex $v$ and each sequence of three consecutive neighbors $u_1,u_2,u_3$ of $v$ in clockwise order around $v$ such that $u_2$ is high, $v$ sends charge $\tfrac7{16}$ to each of $u_1$ and $u_3$.

\item Every low vertex of degree at least 5 sends charge $\tfrac14$ to each neighbor of degree 4.
\end{enumerate}

We now analyze the charge of each vertex and each face after the discharging procedure.

Every face of degree 3 has initial charge $0$, and since it is not involved in any discharging rules, the final charge is also $0$. 
Every face $f$ of degree $d\ge 4$ starts with charge $2d-6$ and sends at most $\tfrac{d}{4}$ by (R1). The final charge of $f$ is therefore at least $2d-6-\tfrac{d}{4}\ge 1$.

Let $v$ be a vertex of degree 4. Then $v$ receives charge $\tfrac78$ from each of its (at least) two high neighbors by (R2). 
If $v$ either has another neighbor of degree at least 5 or is incident to a face of degree at least 4, then $v$ receives an additional charge of $\tfrac14$ and its final charge is therefore at least $-2+\tfrac78+\tfrac78+\tfrac14=0$. 
Otherwise, we can assume that $v$ is adjacent to precisely two high vertices $u_1,u_3$ and two vertices $u_2, u_4$ of degree 4, and is incident to four triangular faces. 
Note also that if $vu_1u_3$ is a face of $G$, then by rule (R3) $v$ receives an additional charge of $\tfrac{7}{16}$ and therefore its final charge is $-2+\tfrac78+\tfrac78+\tfrac7{16}\ge \tfrac3{16}$. 

As a consequence, we can assume without loss of generality that the faces incident with $v$ are $vu_1u_2$, $vu_2u_3$, $vu_3u_4$, and $vu_4u_1$. 
It follows from Claim~\ref{cl:4tri} that $u_1$ and $u_3$ are adjacent and $u_2$ and $u_4$ are adjacent. 
Recall that the embedding of $G$ is edge-maximal, and thus there is an ordering of the neighbors of $u_1$ such that any two consecutive vertices in the ordering are adjacent. 
Since $u_1$ has more than 4 neighbors, it follows that at least one of $v,u_2,u_4$ is adjacent to a vertex not in  $\{v,u_1,u_2,u_3,u_4\}$, a contradiction.

Let $v$ be a low vertex of degree $d\ge 5$. Then $v$ receives charge $\tfrac78$ from each of its (at least) two high neighbors by (R2), and sends at most $(d-2)\tfrac14$ by (R4). 
Its final charge is therefore at least $d-6+{7\over 8}+{7\over 8}-(d-2)\tfrac14\ge 0$.

Finally, let $v$ be a high vertex. Then $v$ sends charge at most $\tfrac78d$ by (R2) and (R3), so its new charge is at least $d-6-\tfrac78d=\tfrac{d}8-6\ge \tfrac{K-44}8$.

We proved that each vertex and face has nonnegative charge, and each high vertex has charge at least $\tfrac{K-44}8$. 
Since there are at least $K+4$ high vertices, the total charge (which equals $6g-12$) is at least $(K+4)\tfrac{K-44}8$. We obtain $K^2-40K-48g-80\le 0$, and this contradicts our choice of $K$ since $K$ satisfies $K^2-40K-48g-80=1$.
\end{proof}

\section{Triangle-free graphs on surfaces}

\begin{thm}\label{tri-free}
Every triangle-free graph embeddable on a surface of Euler genus $g$ is $(0, 0, K)$-colorable where $K=K(g)=\lceil{10g+32\over 3}\rceil$.
\end{thm}
\begin{proof}
Assume for the sake of contradiction that there is a triangle-free graph $G$ embeddable on a surface of Euler genus $g$ that is not $(0,0,K)$-colorable. 
We choose $g$ minimum, and with respect to this, we choose $G$ such that the sum of the number of vertices and the number of edges is minimum. 
By the minimality of $g$ we may assume that $G$ is cellularly embedded on a surface of Euler genus $g$ (see~\cite[Propositions 3.4.1 and 3.4.2]{MoTh}) and from now on, we fix this embedding. 

By the minimality of $G$, we can also assume that $G$ is connected and has minimum degree at least 3.
A \emph{high} and \emph{low} vertex is a vertex of degree at least $K+3$ and at most $K+2$, respectively. 
A $4^+$-vertex that is not high is a {\it medium} vertex. 
By Lemma~\ref{lem:vx-degree}, every low vertex has at least one high neighbor.

We will also assume that for a (partial) $(0, 0, K)$-coloring $\varphi$ of $G$, the three colors will be $a, b, k$ and the graph induced by the color $a$, $b$, $k$ has maximum degree at most $0, 0, K$, respectively. 

%

\begin{claim}\label{vx-3}
Every $3$-vertex in $G$ that is adjacent to at least two $3$-vertices is incident to a $5^+$-face. 
\end{claim}
\begin{proof}
Assume for the sake of contradiction that there is a $3$-vertex $v$ that is adjacent to two $3$-vertices and is incident to only $4$-faces. 
Note that $v$ cannot be adjacent to three $3$-vertices since it must be adjacent to a high vertex by Lemma~\ref{lem:vx-degree}. 
Let $u_1, u_2, u_3$ be the neighbors of $v$ where $u_2$ and $u_3$ are $3$-vertices. 
Also, for $i\in\{1, 2, 3\}$, let $w_i$ be the neighbor of $u_i$ so that $v, u_i, w_i, u_{i+1}$ are the vertices incident with a $4$-face in this order (where $u_4=u_1$). 
See Figure~\ref{fig-bad3vx} (where the white vertices do not have incident edges besides the ones drawn, and the black vertices may have other incident edges).
It is easy to check that $v, u_1, u_2, u_3, w_1, w_2, w_3$ must be all
distinct vertices, since $v, u_2, u_3$ are $3$-vertices, $u_1$ has
degree at least 3, $G$ has no $3$-cycles, and all faces incident
to $v$ are 4-faces. 
Also, $u_2$ and $w_3$ have no common neighbors, since that would create a $3$-cycle. 
Since $u_2$ and $w_3$ have no common neighbor, removing $v$ and adding the edge $u_2w_3$ results in a smaller graph $H$ that has no $3$-cycles and is embeddable on the same surface. 
Thus, $H$ has a $(0, 0, K)$-coloring $\varphi$.

We will extend this coloring of $H$ to $G$ to obtain a contradiction. 
If $\{\varphi(u_1), \varphi(u_2), \varphi(u_3)\}\ne \{a,b,k\}$, then we can use the missing color on $v$ to extend the coloring. 
Moreover, it must be that $\varphi(u_1)=k$, otherwise we could color $v$ with the color $k$ since $K\geq 4$. 
We know that $\varphi(w_2)=k$ since $\{\varphi(u_2), \varphi(u_3)\}=\{a, b\}$.
Also, $\varphi(w_3)=k$ since $\{\varphi(u_2), \varphi(u_3)\}=\{a, b\}$ and $u_2w_3$, $u_3w_3$ are edges in $H$.
Now we can color $v$ with $\varphi(u_3)$ and recolor $u_3$ with $\varphi(u_2)$.
This is a $(0, 0, K)$-coloring of $G$, which is a contradiction. 
\end{proof}

\begin{figure}[h]
	\begin{center}
  \includegraphics[scale=0.9]{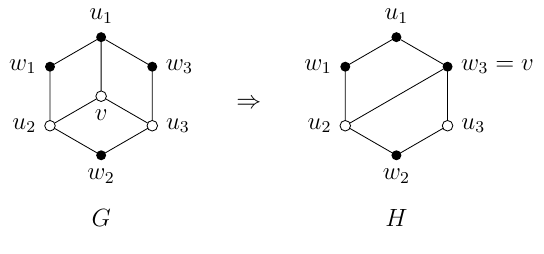}
  \caption{Obtaining $H$ from $G$ in Lemma~\ref{vx-3}}
  \label{fig-bad3vx}
	\end{center}
\end{figure}

By Lemma~\ref{lem:num-high}, $G$ contains at least three high vertices. 
By using the fact that the graph is triangle-free, we can guarantee more high vertices. 

\begin{claim}\label{num-high}
There are at least six high vertices in $G$. 
\end{claim}
\begin{proof}
Let $S$ be the set of high vertices in $G$, and assume for the sake of
contradiction that $|S|\le 5$. If $S$ induces a bipartite subgraph of
$G$, we color $S$ properly with colors $a$ and $b$. Otherwise, since
$G$ is triangle-free, it follows that $S$ induces a 5-cycle $C$. In this
case, we color the vertices of $C$ using colors $a,b,a,b,k$, in this order, and call
$v$ the unique vertex of $C$ colored $k$. Let $N$ be the set of
neighbors of $v$ not in $C$. Since $G$ is triangle-free, $N$ is a
stable set and each vertex $u\in N$ has
at most one neighbor in $C$ distinct from $v$. It follows that the
coloring of $C$ can be properly extended to $N$ by assigning colors
$a$ and $b$ only.

We now complete the coloring of $G$ greedily (by considering the
uncolored vertices in an arbitrary order) as follows: if $w$ has no
neighbor colored $a$ or $b$, then assign the free color to
$w$. Otherwise, assign color $k$ to $w$. Note that each vertex that
has been colored $k$ during the greedy coloring has degree at most $K+2$, and at least one
neighbor colored $a$ and one neighbor colored $b$. Therefore, it has
at most $K$ neighbors colored $k$. This shows that $G$ is
$(0,0,K)$-colorable, which is a contradiction.
\end{proof}


We will use the discharging procedure laid out in Subsection~\ref{subsec:discharging}.
For a vertex $v$ and a face $f$ of $G$, the initial charge is $d(v)-4$ and $d(f)-4$, respectively. 
The initial charge sum is $4g-8$ by Euler's formula. 

Here are the discharging rules:

\begin{enumerate}[(R1)]

\item Each high vertex sends charge $4\over 5$ to each adjacent vertex.

\item Each medium vertex sends charge ${1\over 5}$ to each adjacent $3$-vertex. 

\item Each $5^+$-face sends charge $1\over 5$ to each incident $3$-vertex. 

\end{enumerate}

The discharging rules (R1) and (R2) indicate how the vertices send their charge to adjacent vertices. 
Rule (R3) is the only rule where a face is involved. 




We now analyze the charge of each vertex and each face after the discharging procedure.

Let $f$ be a face. 
Since $G$ has no $3$-cycles, the length of $f$ is at least $4$. 
If $f$ is a $4$-face, then no rule applies to $f$, thus $\mu^*(f)=\mu(f)=d(f)-4=0$. 
If $f$ is a $5^+$-face, then (R3) is the only rule that applies to $f$, and therefore $\mu^*(f)\geq d(f)-4-{d(f)\over 5}={4d(f)\over 5}-4\geq 0$. 

If $v$ is a $3$-vertex, then $\mu(v)=d(v)-4=-1$.
Also, $v$ receives charge ${4\over 5}$ by (R1), since $v$ is adjacent to a high vertex by Lemma~\ref{lem:vx-degree}, 
By Lemma~\ref{vx-3}, $v$ is either adjacent to at most one $3$-vertex or incident to a $5^+$-face. 
If $v$ is adjacent to at most one $3$-vertex, then $v$ receives either an additional charge of ${4\over 5}$ by (R1) or charge ${1\over 5}$ by (R2).
Thus, $\mu^*(v)\geq -1+{4\over 5}+{1\over 5}=0$.
If $v$ is incident to a $5^+$-face, then $v$ receives charge ${1\over 5}$ by (R3).
Thus, $\mu^*(v)\geq -1+{4\over 5}+{1\over 5}=0$.

If $v$ is a medium vertex, then $v$ receives charge ${4\over 5}$ by (R1) since $v$ is adjacent to a high vertex by Lemma~\ref{lem:vx-degree}. 
Also, by (R2), $v$ sends charge ${1\over 5}$ to each adjacent $3$-vertex. 
Thus, $\mu^*(v)\geq d(v)-4+{4\over 5}-{d(v)-1\over 5}={4d(v)-15\over 5}> 0$.

%
If $v$ is a high vertex, then it sends charge ${4\over 5}$ to each neighbor.
Thus, 
$\mu^*(v)
=d(v)-4-{4d(v)\over 5}
={d(v)\over 5}-4
\geq {K(g)+3\over 5}-4
={\lceil{10g+32\over 3}\rceil+3\over 5}-4
\geq {{10g+32\over 3}+3\over 5}-4
={10g-19\over 15}
>{2g-4\over 3}$. 
Thus, each high vertex $v$ has final charge greater than $2g-4\over 3$. 


According to Claim~\ref{num-high}, there are at least six high vertices. 
Since each high vertex has final charge greater than ${4g-8\over 6}$ and every other vertex and face has nonnegative final charge, the sum of the final charge is greater than $4g-8$.
This is a contradiction since the initial charge sum was $4g-8$. 
Therefore, a counterexample to Theorem~\ref{tri-free} does not exist.
\end{proof}

\subsection{Tightness example}\label{subsection-tight}

In this subsection, we will show that the growth rate of $K(g)$ in Theorem~\ref{tri-free} is tight by constructing, for some constant $\epsilon>0$ and infinitely many values of $g$, a triangle-free graph that is embeddable on a surface of Euler genus $g$ but is not $(0, 0, \lceil\epsilon g\rceil)$-colorable.
We will actually do better and construct a graph with girth $6$ that
is not $(0, 0, \lceil\epsilon g\rceil)$-colorable. Our construction is inspired by a
classical construction of Blanche Descartes~\cite{Des54}.

Given a set $S$ of seven vertices in a graph that are pairwise distance at least $3$ apart from each other, let ``adding $C_7$ to $S$'' mean that you add a disjoint copy of $C_7$ and add a perfect matching between the seven new vertices and vertices in $S$. 
Note that this operation does not create $3$-, $4$-, or $5$-cycles. 

Now, construct $H_k$ by starting with seven disjoint copies $D_1, \ldots, D_7$ of $C_7$.
For every set $\{v_1, \ldots, v_7\}$ of seven vertices where $v_i\in D_i$ for $i\in\{1, \ldots, 7\}$, do the operation of ``adding $C_7$ to $\{v_1, \ldots, v_7\}$'' $7k+1$ times.

For every $(0, 0, k)$-coloring of $H_k$, there is a vertex $u_i$ colored with the third color in each $D_i$ for $i\in \{1, \ldots, 7\}$, since a $7$-cycle cannot be properly colored with two colors. 
Now consider the $7k+1$ copies of $C_7$ added to $\{u_1, \ldots, u_7\}$. 
Since each vertex $u_i$ is adjacent to at most $k$ vertices of the third color, there must exist a copy of $C_7$ where none of the vertices are colored with the third color, a contradiction.
Hence, $H_k$ is not $(0, 0, k)$-colorable. 

Note that $H_k$ has $7^8 (7k+1)+49$ vertices and $2\cdot 7^8 (7k+1)+49$ edges, and therefore $H_k$ has Euler genus at most $7^8 (7k+1)+1$.
Hence, $H_k$ is a graph with girth $6$ that is embeddable on a surface of Euler genus at most $7^8 (7k+1)+1$ and is not $(0, 0, k)$-colorable. It follows that there is a constant $\epsilon>0$ and infinitely many values of $g$, for which we can construct a graph of girth 6 that is embeddable on a surface of Euler genus $g$ but is not $(0, 0, \lceil\epsilon g\rceil)$-colorable.

\section{Graphs of girth at least 7 on surfaces}

\begin{thm}\label{thm:g7}
Every graph of girth at least 7 embeddable on a surface of Euler genus $g$ is $(0, K)$-colorable where $K=K(g)=5+\lceil\sqrt{14g+22}\rceil$. 
\end{thm}

\begin{proof}

Assume for the sake of contradiction that there is a graph $G$ with girth at least $7$ embeddable on a surface of Euler genus $g$ that is not $(0,K)$-colorable. We choose $g$ minimum, and with respect to this, we choose $G$ such that the sum of the number of vertices and the number of edges is minimum. By the minimality of $g$ we may assume that $G$ is cellularly embedded on a surface of Euler genus $g$ (see~\cite[Propositions 3.4.1 and 3.4.2]{MoTh}) and from now on, we fix this embedding. 

By the minimality of $G$, we can also assume that $G$ is connected and has minimum degree at least 2.
A \emph{high} and \emph{low} vertex is a vertex of degree at least $K+2$ and at most $K+1$, respectively. 
By Lemma~\ref{lem:vx-degree}, every non-high vertex is adjacent to at least one high vertex.
By Lemma~\ref{lem:num-high}, $G$ contains at least two high vertices. 
By using the fact that the graph has girth at least $7$, we can guarantee more high vertices. 

We will also assume that for a (partial) $(0, K)$-coloring $\varphi$ of $G$, the two colors will be $0$ and $k$, and the graph induced by the color 0 and $k$ has maximum degree at most $0$ and at most $K$, respectively. 

\begin{claim}
There are at least $K+2$ high vertices. 
\end{claim}
\begin{proof}
Assume for the sake of contradiction that the set $H$ of high vertices has size at most $K+1$. 
First color all the vertices of $H$ with the color $k$.
Let $M$ be the set of vertices not in $H$ that have at least one neighbor in $H$, and let $S$ be a maximum independent set in $M$. 
Now color all vertices of $S$ with the color $0$ and color all vertices of $M-S$ with the color $k$. 
For the remaining vertices, we proceed by a greedy algorithm: if a vertex $v$ has a neighbor colored $0$, then use color $k$ on $v$, otherwise, use color $0$ on $v$. 

We now show that this coloring is indeed a $(0, K)$-coloring of $G$. 
For a vertex $v$ in $H$, the neighbors of $v$ that are colored with $k$ are  partitioned into two sets $T_1$ and $T_2$ where $T_1\subseteq H$ and $T_2\subseteq M$. 
Consider a vertex $u\in T_2$. It follows from the definition of $S$ that $u$ is adjacent to a vertex $u_1$ in $M$ that is colored $0$.
This vertex $u_1$ must have a neighbor $u_2$ in $H$, since $u_1$ is in $M$. 
Moreover, since $G$ has girth at least $7$, we know that $u_2\not\in T_1$ and for any two vertices $u,w\in T_2$, we have $u_2=w_2$ if and only if $u=w$. 
Therefore the number of neighbors of $v$ that are colored with $k$ is at most $|T_1|+|T_2|\leq |H|-1\leq K$. 

A vertex in $V(G)-H$ that is colored with $k$ must be adjacent to a
vertex of color $0$, and thus has at most $K$ neighbors colored with
$k$. It is easy to check that no vertex in $V(G)-H$ that is colored
with $0$ has a neighbor colored with $0$. 
Hence, we obtain a $(0,K)$-coloring of $G$, which is a contradiction. It follows that there are at least $K+2$ high vertices. 
\end{proof}

We will use the discharging procedure laid out in Subsection~\ref{subsec:discharging}.
For a vertex $v$ and a face $f$ of $G$, the initial charge is $5d(v)-14$ and $2d(f)-14$, respectively. 
The initial charge sum is $14g-28$ by Euler's formula. 

Here is the unique discharging rule:

\begin{enumerate}[(R1)]

\item Every high vertex $v$ sends charge 4 to each of its neighbors

\end{enumerate}

We now analyze the charge of each vertex and each face after the discharging procedure.

Observe that the charge of a face remains the same, and since $G$ has girth at least 7, all faces have nonnegative final charge. 
A non-high vertex $v$ starts with initial charge $5d(v)-14\ge -4$ and receives a charge of 4 from each of its (at least one) high neighbors, and therefore the final charge of $v$ is also nonnegative. 
Finally, since a high vertex $v$ sends a charge of 4 to each of its neighbors, its final charge is $5d(v)-14-4d(v)=d(v)-14\ge K+2-14=K-12$.

Consequently, the total charge $14(g-2)$ is at least $(K+2)(K-12)$. 
This is equivalent to $K^2-10K+4-14g\leq 0$,
 which contradicts the definition of $K$ ($K$ satisfies $K^2-10K+4-14g> 0$).
\end{proof}

\subsection{Tightness example}
We now prove that the bound on $K(g)$ in the statement of
Theorem~\ref{thm:g7} is best possible, up to a multiplicative constant factor.
We construct, for some constant $c>0$ and infinitely many values of $g$, a graph of girth at least 7 embeddable on a surface of Euler genus $g$, with no $(0,K)$-coloring where $K=K(g)=\lfloor c\sqrt{g}\rfloor$.

In a $(0, K)$-coloring, let $0$ and $k$ be the two colors where the vertices of color $0$ and $k$ induce a graph of maximum degree at most $0$ and $K$, respectively. 
A \emph{2-star} is obtained from a star by subdividing every edge once. Take a 2-star with $3K+2$ leaves, and for any two leaves $u$ and $v$, add an edge between $u$ and $v$ and then subdivide this edge exactly twice (in other words, replace it by a path on 3 edges).
Let $S_K$ be the resulting graph. Now, take two copies of $S_K$, and
join their centers by an edge (see Figure~\ref{fig-g7non0k} for the
case $K=1$). The resulting graph has $
(3K+2)(6K+6)+2$ vertices and $(3K+2)(9K+7)+1$ edges, and therefore has Euler genus at most $(3K+2)(3K+1)=9K^2+9K+2$.

At least one of the two centers is colored with $k$. 
Consider the corresponding copy of $S_K$. At least $2K+2$ of the neighbors of the center (in the copy of $S_K$) are colored with $0$. 
The corresponding $2K+2$ leaves of the 2-star are then colored with $k$. 
Let $L$ be the set of these leaves, and let $D$ be the sum, over all vertices $v$ of $L$, of the number of neighbors of $v$ colored $k$. 
Observe that in each added path on 3 edges, at least one of the newly added vertices is colored with $k$, so each added path on 3 edges between two vertices of $L$ contributes at least 1 to $D$. Since there are $|L|(|L|-1)/2$ such paths, at least one of the vertices of $L$ has at least $(|L|-1)/2$ vertices colored $k$.
If $(|L|-1)/2>K$, then this is a contradiction.

It follows that there is a constant $c>0$ and infinitely many values of $g$, for which we can construct a graph of girth at least 7 embeddable on a surface of Euler genus $g$, with no $(0,\lfloor c\sqrt{g}\rfloor)$-coloring.

\begin{figure}[ht]
	\begin{center}
  \includegraphics[scale=0.9]{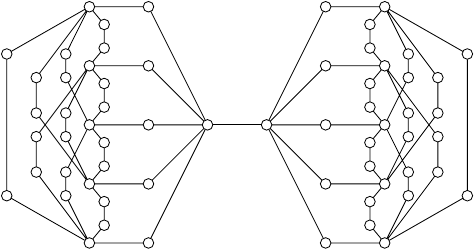}
  \caption{A construction of a graph with girth 7 that is not $(0, K)$-colorable.}
  \label{fig-g7non0k}
	\end{center}
\end{figure}

\section{Open problems}

A natural question is to find a version of Theorem~\ref{thm:g7} for
graphs of arbitrary large girth. A slight variation of the proof of
Theorem~\ref{thm:g7} easily shows that a graph of girth at least
$\ell$ embeddable on a surface of Euler genus $g$ is $(0,
O(\sqrt{g/\ell}))$-colorable, where the hidden constant depends on
neither $g$ nor $\ell$. In an early version of this manuscript, we conjectured the following stronger statement.

\begin{conj}\label{conj}
There is a function $c=o(1)$ such that any graph of girth at least $\ell$ embeddable on a surface of Euler genus $g$ is $(0, O(g^{c(\ell)}))$-colorable.
\end{conj}

Note that a graph that is $(0,k)$-colorable has a proper coloring with
$k+2$ colors (since a graph with maximum degree $k$ has a proper
$(k+1)$-coloring). As a consequence, the following result of Gimbel
and Thomassen~\cite{GT97} gives a lower bound of the order
$\tfrac1{2\ell+2}$ on such a function $c$.

\begin{thm}[\cite{GT97}]
For any $\ell$, there exist a constant $c>0$ such that for arbitrarily small $\epsilon>0$ and sufficiently large $g$, there are graphs of girth at least $\ell$ embeddable on surfaces of Euler genus $g$ that have no proper coloring with less than $c\, g^{\tfrac{1-\epsilon}{2\ell+2}}$ colors.
\end{thm} 

It was subsequently observed by Fran\c cois Dross that an argument similar to that
of the proof of Theorem~\ref{thm:g7} shows that if the girth is at
least $6t+1$, then there are $K^t$ vertices of degree at
least $K$; we just have to consider paths of length $3t$ starting from
some vertex $v$. Using a similar computation as in the proof of
Theorem~\ref{thm:g7}, this shows that any graph of girth at least
$\ell$ embeddable on a surface of Euler genus $g$ is $(0,
O(g^{6/(\ell+5)}))$-colorable, which proves Conjecture~\ref{conj}.

\section*{Acknowledgments}

The authors would like to thank Gwena\"el Joret for the interesting
discussions on short non-contractible cycles, Fran\c cois Dross for
allowing us to mention his remark about Conjecture~\ref{conj}, and a
reviewer for the excellent suggestions.

\end{document}